\documentclass[12pt,reqno]{amsart}
\usepackage{amssymb}

\newcommand{\e}{\varepsilon}

\newcommand{\la}{\lambda}
\newcommand{\al}{\alpha}

\newcommand{\fy}{\varphi}

\newcommand{\p}{\partial}
\newcommand{\I}{\infty}

\newcommand{\R}{\mathbb{R}}

\newcommand{\N}{\mathbb{N}}

\numberwithin{equation}{section}

\newtheorem{thm}{Theorem}[section]

\newtheorem{lem}[thm]{Lemma}

\theoremstyle{remark}

\newcommand{\ran}{\rangle}
\newcommand{\lan}{\langle}

\newcommand{\weak}[1]{{\text{w-}#1}}
\newcommand{\lec}{\lesssim}
\newcommand{\gec}{\gtrsim}

\newcommand{\EQ}[1]{\begin{equation} \begin{split} #1
 \end{split} \end{equation}}

\newcommand{\Del}[1]{}

\newcommand{\pt}{&}
\newcommand{\pr}{\\ &}
\newcommand{\pq}{\quad}
\newcommand{\pn}{}

\newcommand{\prq}{\\ &\quad}

\newcommand{\LR}[1]{{\lan #1 \ran}}
\newcommand{\de}{\delta}

\newcommand{\ka}{\kappa}
\newcommand{\ga}{\gamma}

\newcommand{\na}{\nabla}

\renewcommand{\th}{\theta}

\newcommand{\De}{\Delta}
\newcommand{\Om}{\Omega}

\newcommand{\IN}[1]{\text{ in }#1}

\newcommand{\NN}{\mathcal{N}}

\newcommand{\m}{m}

\author{L.~ Aloui}
\address{Department of Mathematics, University of Bizerte, Bizerte, Zarzouna 7021, Tunisia.}
\email {\it Lassaad.Aloui@fsg.rnu.tn}
\thanks{L.~A.~is partially  supported  by LAMSIN}

\author{S.~Ibrahim}
\address{Department of Mathematics and Statistics,\\University of Victoria\\
 PO Box 3060 STN CSC\\   Victoria, BC, V8P 5C3\\ Canada}
\email{ibrahim@math.uvic.ca} \urladdr{
http://www.math.uvic.ca/~ibrahim/}
\thanks{S.~I.~is partially supported by NSERC\# 371637-2009 grant and a start up fund from University of Victoria}

\author{K.~Nakanishi}
\address{Department of Mathematics, Kyoto University} \email {\it n-kenji@math.kyoto-u.ac.jp}
\thanks{K.N.~is partially supported by Grant-in-Aid for Scientific Research 21740095}

\title[Exponential energy decay for damped NLKG]{Exponential energy decay for \\ damped Klein-Gordon equation with \\ nonlinearities of arbitrary growth}
\date{\today}

\begin{document}
\begin{abstract}
We derive a uniform exponential decay of the total energy for the nonlinear Klein-Gordon equation with a damping around spatial infinity in $\R^N$ or in the exterior of a star-shaped obstacle.  
Such a result was first proved by Zuazua \cite{Z,Zua91} for defocusing nonlinearity with moderate growth, and later extended to the energy subcritical case by Dehman-Lebeau-Zuazua \cite{DLZ}, using linear approximation and unique continuation arguments.
We propose a different approach based solely on Morawetz-type {\it a priori} estimates, which applies to defocusing nonlinearity of arbitrary growth, including the energy critical case, the supercritical case and exponential nonlinearities in any dimensions. 
One advantage of our proof, even in the case of moderate growth, is that the decay rate is independent of the nonlinearity. 
We can also treat the focusing case for those solutions with energy less than the one of the ground state, once we get control of the nonlinear part in Morawetz-type estimates. 
In particular this can be achieved when we have the scattering for the undamped equation.
\end{abstract}



\maketitle

\tableofcontents 


\section{Introduction}
We study total energy decay for the damped nonlinear Klein-Gordon equation (NLKG) for 
$u(t,x):[0,\I)\times\Om\to\R$, 
\EQ{\label{dNLKG}
 u_{tt} + a(x) u_t -\De u + u + f'(u)=0 \pq (t,x)\in [0,\I)\times\Om,\\
 u = 0 \pq (t,x)\in(0,\I)\times\p\Om,}
where $\Om$ is a $C^1$ (exterior) domain in $\R^N$ ($N\ge 1$), the damper $a:\Om\to[0,\I)$ is effective uniformly around the spatial infinity 
\EQ{ \label{asm a}
 a(x)\ge 0, \pq a \in L^\I(\Om), \pq  \liminf_{|x|\to\I} a(x)> 0,}
and the nonlinear energy $f:\R\to\R$ is $C^2$ satisfying\footnote{These assumptions are quite natural for the nonlinear Klein-Gordon energy. $f(0)=0$ is necessary to have finite energy, $f'(0)=0$ follows from the evenness, while $f''(0)=0$ means that the mass term is exactly $u$. It can be replaced with $mu$ ($m>0$), but in the massless case $m=0$ we do not have the exponential decay. See e.g. \cite{Nakao} for the algebraic decay in the massless case.} 
\EQ{
 f(u)=f(|u|), \pq f(0)=f'(0)=f''(0)=0.} 
This equation may be regarded as a simple model for nonlinear dissipative hyperbolic equations with purely dispersive spatial regions. The stabilization problem consists in deriving sufficient conditions on the dissipation $a(x)$ 
guaranteeing uniform decay of the total energy, defined by  
\EQ{ \label{energy}
 E(u;t) :=\int_\Om |u_t|^2 + |\na u|^2 + |u|^2 + 2 f(u) dx.}
When the nonlinear energy is not positive definite in the whole energy space, we have to distinguish it from the free energy 
\EQ{\label{free energy}
 E_F(u;t) := \int_{\R^N} |\dot u|^2 + |\na u|^2 + |u|^2 dx.}
Multiplying \eqref{dNLKG} by $\dot u$ and integrating over
the slab $(t_{1},t_{2})\times\Om$, one formally gets the energy identity 
\EQ{ \label{decay} 
 E(u;t_{2})-E(u;t_{1})
  =-2\int_{t_{1}}^{t_{2}}\int_{\Om}a(x)|\dot u(t,x)|^{2}dxdt.}
Heuristically, one achieves the stabilization when certain portion of the energy keeps flowing into the kinetic part on the region of effective dissipation. Then one typically obtains exponential decay in the form
\EQ{ \label{unif dec 0}
 E(u;t) \le Ce^{-\ga t}E(u;0),}
with some constants $C,\ga>0$. 

The two major factors in this problem are the geometric structure of the domain, the boundary condition and the dissipation on one hand, and the nonlinear interaction on the other hand. 
The geometric aspects have been extensively studied especially in the linear setting, with deep insights into wave propagation properties, see e.g., \cite{Lions,RauchTaylor,BLR}. 

However the nonlinear interaction becomes the bigger obstruction for the finer investigation of the wave propagation. 
The standard strategy is to start with some fairly weak decay properties which are accessible even in the nonlinear setting, and then through some limiting procedures and approximation arguments, to reduce the problem to the linear case, where much finer informations are available.  
In this type of arguments, the linear approximation poses essential limitations on the generality of the nonlinear interaction. 
We refer to \cite{Z,Zua91,D,DLZ,DG,AC} for the stabilization results in the current setting \eqref{dNLKG}.

\subsection{Defocusing case}
Our observation is that one can skip the latter part of the argument, i.e. the linear approximation, at least in the simplest geometric setting, thereby obtaining the exponential decay for fairly general nonlinearity with the repulsive sign. 
Our nonlinearity condition is a sort of self-coerciveness, given in terms of 
\EQ{ \label{the g}
 g(u) := uf'(u) - 2f(u).}
The first main result of this paper is 
\begin{thm}[defocusing case]
\label{Mainresult1} 
Let $\Om\subset\R^N$ ($N\in\N$) be either the whole space or the exterior of a star-shaped obstacle with a $C^1$ boundary.  
Assume that $a:\Om\to\R$ satisfies for some constants $M,R,a_0>0$ 
\EQ{ \label{asm a prec}
  0\le a \le M, \pq \inf_{|x|>R}a(x)\ge a_0,}
$f:\R\to\R$ is $C^2$ satisfying $f(0)=f'(0)=f''(0)=0$, $f(u)=f(|u|)$, and 
\EQ{\label{condition f} 
  g(u)\ge 0, \pq 0\leq f(u) \le C_0(|u|^2+g(u)),}
for some constant $C_0>0$. Then there exists $\ga>0$ determined only by $N$, $M$, $R$, $a_0$ and $C_0$, such that  
for every initial data $(u(0,x),u_t(0,x))$ with finite energy $E(u;0)<\I$, there exists a global weak solution $u$ of \eqref{dNLKG} satisfying 
\EQ{ \label{exp dec}
 E(u;t) \le 2e^{-\ga t}E(u;0)}
for all $t\ge 0$. If \eqref{condition f} is weakened to 
\EQ{\label{condition f2}
 g(u) \ge 0, \pq 0\le f(u) \le C_0(|u|^2 + |u|^q + g(u))}
with some $2\le q\le 2^\star:=2N/(N-2)$ ($2\le q<\I$ if $N\le 2$), then \eqref{exp dec} still holds, but $\ga>0$ depends (non-increasingly) on $E(u;0)$ and $q$ too.    
\end{thm}
Note that the rate $\ga$ is uniform in the whole energy space under the condition \eqref{condition f}, while under \eqref{condition f2} it is uniform on each ball of the energy space. 
In any dimension $N\geq 1$, examples of the nonlinearity include all defocusing powers and exponentials 
\EQ{ \label{exp1}
 f(u) = \la \exp(\mu|u|^\nu) |u|^{2+\al}, \pq (\la,\mu,\nu,\al \geq 0),}
and their convex combinations. 
The global wellposedness in the energy space is so far available under the growth condition ($H^1$ subcritical or critical case) 
\EQ{ \label{grow p}
 |f'(u_1)-f'(u_2)| \le C[|u_1|+|u_2|]^{p-1}|u_1-u_2|, \pq p+1 \le 2^\star=\frac{2N}{N-2},}
with suitable modifications when $2^\star<3$ ($N>6$), and exponential growth condition for $N=2$ 
(see, e.g.~\cite{GV1,SS,2Dglobal}). 
Once we have global wellposedness, the above uniform decay of course applies to the unique solution,  
but {\it our theorem applies even if we do not know unique existence (i.e.~in the $H^1$ supercritical case)}. 

Let us compare our result with the preceding ones for the stabilization problem of \eqref{dNLKG}.  
 Zuazua in \cite{Z,Zua91} first proved the exponential decay by using the unique continuation result by Ruiz \cite{Ruiz}, in bounded and unbounded domains for nonlinear term $f'$ mapping $H^1$ into $L^2$, i.e.~\eqref{grow p} with $p\le N/(N-2)$, and with $a(x)$ effective on the whole boundary and around the infinity. 
Then Dehman in \cite{D} relaxed the nonlinearity condition by making use of the Strichartz estimates, and propagation of microlocal defect measures, introduced by G\'erard \cite{G1} to describe possible lack of compactness in $H^1$ bounded sequence of linear solutions. 
The latter result was further extended to the full range of $H^1$ subcritical case, i.e.~\eqref{grow p} with $p+1<2^\star$, by Dehman-Lebeau-Zuazua \cite{DLZ}. 
They used a linearization argument enabling them to propagate the possible singularities (lack of compactness) along bi-characteristics to the region where the damper is effective. 
However this linear approximation might fail in the $H^1$ critical case $p+1=2^\star$, see for example G\'erard \cite{G2}. The $H^1$ critical case was finally solved by Dehman-G\'erard \cite{DG} using the nonlinear profile decomposition of Bahouri-G\'erard \cite{BG} instead. 
 
We point out that these proofs do not explicitly provide any bound on the decay rate $\gamma$. 
Recently T\'ebout \cite{Tebou} gave another proof based on a Carleman estimate due to Ruiz \cite{Ruiz}. 
His proof is more quantitative, but also requires the condition $p\leq N/(N-2)$ to bound the nonlinear term in the Carleman estimate. 

Our proof of Theorem \ref{Mainresult1} gives explicit estimates on $\ga$, and moreover it is very elementary, especially compared with the above mentioned arguments. 
It consists mainly of three ingredients: (1) a space-time a priori estimate of the Morawetz type, (2) a weighted Sobolev-type inequality adapted to the Morawetz-type estimate, and (3) an energy equipartition estimate adjusted to the weight of the Morawetz-type estimate. 

The Morawetz space-time integral estimate was originally derived in \cite{Morawetz2} to have some weak decay of local energy for the undamped nonlinear Klein-Gordon equation 
\EQ{ \label{NLKG}
 \ddot u - \De u + u + f'(u) = 0,}
in $\R^3$, and then later applied to the nonlinear scattering theory of the undamped equation by Strauss-Morawetz \cite{StrMora}. 
In this paper we use a modified version of the estimate derived by Nakanishi \cite{Noncoer} for the undamped equation, in the form 
\EQ{
 \iint_{|x|<t/2} \frac{|x\dot u+t\na u|^2}{|(t,x)|^3} + \frac{g(u)}{|t|} dx dt \le CE(u;0),\quad x\in\Omega\in\R^N,\; N\geq1,}
which is available in all spatial dimensions. 
We emphasize that it is essential to use the quadratic term (via the weighted Sobolev inequality), instead of the nonlinear term $g(u)$, in order to get decay estimates independent both of the nonlinearity and of the energy size. 
This idea was first used by Nakanishi in \cite{Noncoer} for the nonlinear scattering theory. 

The energy equipartition property was also used in the preceding works for the stabilization, but the mass energy term $|u|^2$ has prevented them to use it in a completely nonlinear way. Introduction of the Morawetz-type weight as well as the weighted Sobolev inequality enables us to control that problematic term directly. 

However, our reliance on the Morawetz-type estimate brings certain limitations of our approach. 
Since it represents the global dispersive property of the solutions, it seems difficult to use in bounded domains, or more precisely in presence of trapped rays. 
This is a big disadvantage compared with the preceding works, which equally apply to bounded domains (provided that one has enough linear estimates). 

Another, more technical restriction in our theorem is the condition \eqref{condition f} posed on $g$. 
The preceding works require similar conditions to get the global stabilization (namely a decay rate independent of the energy size), but they can also work with weaker conditions such as $uf'(u)\ge 0$ for local stabilization (decay rate dependent of the energy size).


An example of $f$ which satisfies \eqref{condition f2} but not \eqref{condition f} is given by
\EQ{
 f(u)=|u|^2\log(1+|u|^2),}
i.e.~a logarithmic perturbation of a linear term, 
whereas similar perturbations of superlinear powers satisfy \eqref{condition f}, for example 
\EQ{
 f'(u) = u^5\log(1+|u|^2),}
which was considered by Tao in \cite{Tao} as an $H^1$ supercritical case in $\R^3$, getting global solutions for smooth radial initial data (cf. \cite{Roy} for the non radial case). 

\subsection{Focusing case}
In Theorem \ref{Mainresult1} for the defocusing case, our assumption \eqref{condition f} on $f$ is used only to discard nonlinear terms in the Morawetz-type estimates. 
That term can be controlled, if we have the nonlinear scattering for the undamped NLKG \eqref{NLKG}, and a global perturbation argument for the Strichartz norms. The first argument provides us with global bounds on the Strichartz space-time norms of the undamped equation, while the second one allows us to transfer the global Strichartz bounds to the damped equation. The two arguments are indeed satisfied, even for a focusing nonlinearity, if we restrict to the solutions with energy less than the ground state (a solution of the static equation with possibly a modified mass) and inside ``the potential well". 
We should point out, however, that the previous works in the defocusing case \cite{Zua91,D,DLZ,DG} could also extend to the same setting, as far as the nonlinearity satisfies their growth conditions. See \cite{FW} for a result for the wave equation on a bounded domain. 

For simplicity of presentation, we consider only the whole space as the domain, and typical nonlinearities of one of the following examples, though our argument applies to more general nonlinearities.
\begin{enumerate}
\item
$N\ge 1$ and $f$ is $H^1$ subcritical: We assume 
\EQ{\label{f sub}
 f(u) = \la_1 |u|^{p_1}+ \cdots + \la_k |u|^{p_k},}
for some $\la_1,\dots,\la_k>0$ and $2_\star<p_1<\dots<p_k<2^\star$, where $2_\star=2+4/N$ is the $L^2$ critical power and $2^\star=2N/(N-2)$ is the $H^1$ critical one. 
\item $H^1$ critical case. We assume $N\ge 3$ and for some $\la$, 
\EQ{ \label{f crit}
  f(u) = |u|^{2^\star}.}
\item $2D$ exponential case. We assume that $N=2$ and, 
\EQ{\label{f exp}
 \pt f(u) = e^{4\pi|u|^2}-1-4\pi|u|^2-(4\pi|u|^2)^2/2.}
\end{enumerate}
Note that the scattering for the lower critical power $2_\star$ is still open for large data even in the defocusing case, and the subtracted $|u|^4$ term in \eqref{f exp} corresponds to that power. 

Global solutions to \eqref{NLKG} do not always exist in these cases. On the one hand, when the data is small in the energy space, it leads to a global solution which scatters. On the other hand, when the data is sufficiently large, then the solution can blow up in finite time, for example if the energy is negative. The threshold between these two dynamics can be constructed as follows.  

Let $J$ be the static energy functional given by
\EQ{
 J(\fy) := \int_{\R^d} [|\na \fy|^2 + |\fy|^2] dx - 2F(\fy), \pq F(\fy):=\int_{\R^d} f(\fy) dx,}
and denote its derivative with respect to amplification by 
\EQ{
 K(u) := 2\int_{\R^N} |\na u|^2 + |u|^2 - uf'(u) dx.}
Consider the constrained minimization problem 
\EQ{ \label{min J}
 { \m} = \inf\{J(\fy) \mid 0\not=\fy \in H^1(\R^d),\ K(\fy)=0\},}
and define 
\EQ{\label{The Set K}
 \pt {\mathcal K}^+=\{(u_0,u_1)\in H^1(\R^N)\times L^2(\R^N) \mid 
  E(u_0,u_1)<{m},\ K(u_0)\ge 0\}, 
 \pr {\mathcal K}^-=\{(u_0,u_1)\in H^1(\R^N)\times L^2(\R^N) \mid 
  E(u_0,u_1)<{m},\ K(u_0)< 0\}. }
Then for the undamped NLKG
\EQ{\label{NLKG}
 u_{tt} - \De u + u + f'(u)=0,}
the solutions with energy below $m$ are divided as follows: 
\begin{itemize}
\item If $(u(0),u_t(0))\in {\mathcal K}^+$, then the solution $u$ of \eqref{NLKG} is global. 
\item If $(u(0),u_t(0))\in {\mathcal K}^-$, then the solution $u$ of \eqref{NLKG} blows up in finite time. 
\end{itemize}
In the subcritical case, this follows essentially from Payne-Sattinger \cite{PaSa} and the local wellposedness in the energy space. 

Recently, Ibrahim-Masmoudi-Nakanishi \cite{IMN} extended it to the critical and 2D exponential cases. Moreover, they proved that in all the cases (1)-(3), the solutions of \eqref{NLKG} in $\mathcal{K}^+$ scatter as $t\to\pm\I$ with global Strichartz bounds, implementing Kenig-Merle's argument \cite{KM1,KM2} to the equation without scaling invariance. In addition, they gave the following characterization of $m$ (see \cite[Proposition 1.2]{IMN}), which is again well known in the subcritical case: {\it There exist $Q(x)$ and $c\in[0,1]$ such that  
\EQ{ 
  m = J^c(Q) = \int |\na Q|^2 + c|Q|^2 dx - 2F(Q),}
\EQ{ \label{static NLKG} 
 \pt -\De Q + c Q = f'(Q), \pq \|\na Q\|_{L^2}^2 + c\|Q\|_{L^2}^2<\I,}
and $J^c(Q)$ is the minimum among the solutions of \eqref{static NLKG}, i.e. $Q$ is the ground state for \eqref{static NLKG}. $c=1$ in the subcritical case \eqref{f sub} and the exponential case \eqref{f exp}, and $c=0$ in the critical case \eqref{f crit}.} 

Hence transferring the global Strichartz bounds to the damped NLKG, we obtain the following stabilization in the set $K^+$. 
\begin{thm}\label{Mainresult3}
Let $N\ge 1$, $\Om=\R^N$ and $f(u)$ given by either (1), (2) or (3). 
Assume that $a:\R^N\to\R$ satisfies \eqref{asm a prec} with some $M,R,a_0>0$. 
Let $0<E_0<m$. Then there exists $\ga>0$ determined by $E_0$, $f$, $M$, $R$, $a_0$ and $N$ such that for any initial data satisfying $E(u;0)\le E_0$ and $K(u(0))\geq0$, we have a unique global solution $u$ of \eqref{dNLKG} with exponential decay  
\EQ{ 
 \frac{2}{N+2} E_F(u;t) \le E(u;t) \le 2 e^{-\ga t}E(u;0),}
for all $t>0$.   Recall that $E_F$ is the free energy given in \eqref{free energy}.
\end{thm}
In the above Theorem, the dependence of $\ga$ upon $E_0$ is inevitable, since $Q$ is a solution of the damped NLKG without any decay.\\ 

In the sequel, $A\lec B$ or $B\gec A$ means $A\le CB$ for some positive constant $C$.


\section{Proof in the defocusing case: Theorem \ref{Mainresult1}}
In this section, we prove the uniform exponential decay in the defocusing case, i.e.~Theorem \ref{Mainresult1}. 
The proof relies solely on the multiplier argument, without any estimate for the linearized equation. 
However in order to justify some energy-type computations, we first approximate the nonlinearity so that we have the global wellposedness, derive the uniform decay for the solutions of the approximate equations with the same initial data, and then take the weak limit to the original equation. 
Note that it is in general hard to get any information for a given weak solution in the absence of the wellposedness. 
\subsection{Truncation of nonlinearity}
First we reduce to the $H^1$ subcritical case by approximation of the nonlinearity. 
Note that this procedure is not entirely trivial. 
On the one hand, we should not simply cutoff the nonlinearity, since $g$ vanishes when $f'$ is linear. 
In other words, approximate nonlinearity must be at least superlinear. 
On the other hand, we need global wellposedness for the approximate equations to justify our energy-type computations. 
The global wellposedness is available in the energy subcritical and critical cases, for which growth condition should be imposed on $f''$, rather than on $f$ or $f'$.  
Consequently, contrary to the classical methods as in \cite{Segal,Lions,Strauss_WeakSol}, we need two step approximations of the nonlinearity, one from below and another from above. 

Let $V(u)=f(u)/|u|^2$ be the nonlinear potential energy. 
Then we have $f(u)=V(u)|u|^2$ and $g(u)=V'(u)|u|^2u$, so our conditions on $f$ are rewritten as 
\EQ{ \label{weak}
 V(0)=0, \pq uV'(u)\ge 0, \pq V(u) \le C_0(1 + uV'(u)),}
and the weaker condition \eqref{condition f2} is given by
\EQ{
 V(u) \le C_0(1 + uV'(u) + |u|^{q-2}).}
We construct a sequence of approximate nonlinear potentials $V_k$ as follows. 
First we fix $\th\in(0,1)$, depending only on the dimension $N$, such that $2+\th<2^\star$. 
If there exists some $k\in\N$ such that for all $|z|\ge k$
\EQ{
 zV'(z) \le kV'(k)|z/k|^{\th},} 
then $f$ is subcritical, i.e. $|f'(z)|\lec |zV'(z)z|+|V(z)z|\lec |z|^{1+\th}$. 
In this case we put $f_k=f$.  
Otherwise, there exists a sequence of $k\in\N$ along which $V'(k)k^{1-\th}$ is increasing. 
For those $k$ we define the first approximation $V_k:\R\to\R$ by 
\EQ{
 V_k(z) = V(z)\ (|z|\le k), \pq zV_k'(z) = \min(zV'(z),kV'(k)|z/k|^{\th}) \ (|z|\ge k),}
and $f_k:\R\to\R$ by
\EQ{
 f_k(z)=V_k(z)|z|^2.}
Then clearly we have $V_k\ge 0$, $V_k'\ge 0$, and $f_k'$ is locally Lipschitz. 
Our choice of $k$ implies that $zV_k'(z)$, $V_k(z)$ and $f_k(z)$ are all increasing in $k$ at each $z$. 
The $V'$ dominance is inherited by $V_k$ as follows, where the constants may depend also on $\th$: for $z>k$ we have 
\EQ{
 V_k(z) 
 \pt\le V(k) + \min(\int_k^z V'(y) dy,\int_k^z V'(k)|y/k|^{\th-1}dy)
 \pr\lec \min(V(z), V(k)+ V'(k)|z|^\th k^{1-\th}) \lec 1+ zV_k'(z),}
and similarly for $z<-k$, while it is trivial for $|z|\le k$. \eqref{condition f2} is inherited in the same way. 

Although the growth for $f_k$ and $f_k'$ is $H^1$ subcritical, $f_k''$ might not be $H^1$ subcritical when $f''$ is oscillating for large $u$. 
Hence to have global wellposedness for approximate equations, we need further approximation $V_{kl}:\R\to\R$ defined for $l>k$ by 
\EQ{
 V_{kl}(z) = V_k(z) \ (|z|\le l) \pq zV_{kl}'(z) = lV_k'(l)|z/l|^{\th} \ (|z|\ge l),}
and $f_{kl}:\R\to\R$ by $f_{kl}(z)=V_{kl}(z)|z|^2$. Then $f_{kl}'$ is locally Lipschitz and 
\EQ{
 |f_{kl}'(z_1)-f_{kl}'(z_2)| \lec (|z_1|+|z_2|)^\th |z_1-z_2|.}
Now that the approximate nonlinearity is $H^1$ subcritical (in fact we can make it as close to be linear as we want), we have the global wellposedness for 
\EQ{ \label{Truncated eq2}
 \p^2_tu + a(x) \p_t u - \De u + u  + f_{kl}'(u)=0.}

Now let $u_{kl}$ be the global solution for each $k<l$ with the same initial data, and assume that we have 
\EQ{ \label{asm unif dec}
 E_{kl}(u_{kl};t) \le 2 e^{-\ga t} E_{kl}(u_{kl};0),}
for some $\ga>0$, all $t>0$, and all $k<l$, where $E_{kl}$ denotes the energy for the approximate equation 
\EQ{
 E_{kl}(v;t) = \int_{\Om} |\dot v|^2 + |\na v|^2 + |v|^2 + f_{kl}(v) dx.}
Then $u_{kl}$ is uniformly bounded in $H^1\times L^2$, and so the classical weak compactness argument (cf. \cite{Strauss_WeakSol,Struwe}) gives us a subsequence, which we still denote by $u_{kl}$, converging as $l\to\I$ to some function $u_k(t,x)$ in the sense 
\EQ{  
 \pt u_{kl}(t) \to u_k(t) \IN{C([0,\I);\weak{H^1_0(\Om)})}, 
 \pr \dot u_{kl}\to\dot u_k\IN{\weak{L^2_{loc}((0,\I);L^2(\Om))}},}
where $\weak{X}$ denotes the space endowed with the weak topology, and also $u_{kl}(t,x)\to u_k(t,x)$ pointwise for almost every $x\in\Om$ at all $t>0$. 
Note that in this limit $l\to\I$, the nonlinear energy $f_{kl}$ and its limit $f_k$ are $H^1$ subcritical, and so are dominated by the $H^1$ norm through the Sobolev embedding.  
Hence the limit $u_k$ is a global weak solution to
\EQ{ 
 \p^2_tu + a(x) \p_t u - \De u + u + f_k'(u)=0,}
and moreover, by Fatou's lemma, the uniform decay for $u_{kl}$, and the Lebesgue dominated convergence theorem respectively, we have
\EQ{
 E_k(u_k;t) \le \liminf_{l\to\I} E_{kl}(u_{kl};t) \le \liminf_{l\to\I} 2e^{-\ga t}E_{kl}(u;0) = 2e^{-\ga t}E_k(u_k;0).} 
Next for a subsequence of $k\to\I$ we have a similar convergence $u_k\to u$ as above, and by Fatou, the uniform decay of $u_k$ and the monotone convergence theorem respectively we get 
\EQ{
 E(u;t) \le \liminf_{k\to\I} E_k(u_k;t) \le \liminf_{k\to\I} 2e^{-\ga t}E_k(u;0) = 2e^{-\ga t}E(u;0),}
and $u$ is a global weak solution of the original damped NLKG. 
Note that we cannot use the dominated convergence theorem in this step, since the nonlinear energy in the limit is no longer controlled by the Sobolev embedding. 

It remains to prove weak convergence of the nonlinear term in the equation. 
Multiplying the equation of $u_{kl}$ with $u_{kl}$ and integrating it in space-time, we get 
\EQ{
 [\LR{\dot u_{kl}|u_{kl}}_{L^2_x}]_0^T + \int_0^T \|\na u_{kl}\|_{L^2_x}^2 - \|\dot u_{kl}\|_{L^2_x}^2 dt 
 + \int_0^T \LR{f_{kl}(u_{kl})|u_{kl}}_{L^2_x} dt =0,}
which implies through the energy identity that 
\EQ{
 \int_0^T \LR{f_{kl}(u_{kl})|u_{kl}}_{L^2_x} dt \lec E(u;0).}
Then by the dominated convergence theorem for the limit $l\to\I$, and by the monotone convergence theorem for the limit $k\to\I$, we get 
\EQ{ \label{uf bound}
 \int_0^T \LR{f_k(u_k)|u_k}_{L^2_x} dt + \int_0^T \LR{f(u)|u}_{L^2_x} dt \lec E(u;0).}
For any $L\gg 1$, let $f^L(u)=\min(f(u),f(L))$ and $f_k^L(u)=\min(f_k(u),f_k(L))$. 
Then for any bounded domain $K\subset\Om$ we have 
\EQ{
 \|f_k(u_k)-f(u)\|_{L^1([0,T]\times K)} 
 \pt\le \|f_k^L(u_k)-f^L(u)\|_{L^1_{t,x}([0<t<T]\times K)} 
 \pr+ \|u_kf_k(u_k)/L\|_{L^1_{t,x}(0<t<T,\ |u_k|>L)} 
 \pr+ \|uf(u)/L\|_{L^1_{t,x}(0<t<T,\ |u|>L)}.}
On the right, the first term tends to $0$ as $k\to\I$ by the dominated convergence theorem, while the other two terms are bounded by $E(u;0)/L$ thanks to \eqref{uf bound}. 
Hence $f_k(u_k)\to f(u)$ in $L^1_{loc}((0,\I)\times\Om)$, and so $u$ is a global weak solution of the original equation  satisfying the uniform exponential decay.

\subsection{Uniform decay for subcritical nonlinearity} \label{defoc sub}
Now it suffices to prove the uniform exponential decay in the $H^1$ subcritical (and defocusing) case, with the decay rate independent of the nonlinear term and the initial data. 
Since we have global wellposedness in this case, energy-type computations used below are easily justified by standard approximation argument. 
In particular we have the energy identity \eqref{decay} for any finite energy solution $u$. 
Denote the energy decrement on the interval $[0,T]$ by 
\EQ{
 A(u;T) := \int_0^T\int_{\Om} a(x)|\dot u|^2 dxdt.} 
We are going to show that for some quantitative $T>0$ and $\de>0$, and for all solutions $u$ with finite energy we have 
\EQ{ \label{energy decrease}
 A(u;T) \ge \de E(u;T).}
Then the energy identity implies that 
\EQ{
 E(u;t+T) \le E(u;t) - \de E(u;t+T),}
for any $t>0$, and hence by repeated use of it we get the exponential decay \eqref{exp dec} with the rate 
\EQ{ \label{exp rate}
 \ga = T^{-1}\log(1+\de).}
 
Hence we are mostly concerned with the case where $A(u;T)$ is small compared with $E(u;0)$. 
The main ingredients of our argument are the energy equipartition property, and the following Morawetz-type estimate, which is a trivial modification of that \cite{Noncoer,2Dsubcrit} for the undamped equation. 
Let 
\EQ{
 \pt \la = |(t,x)| = \sqrt{t^2+|x|^2}, \pq q=\frac{(N-1)}{2\la}+\frac{t^2-|x|^2}{\la^3}, 
 \pr m(u)=\frac{-tu_t+x\cdot\na u}{\la} +uq.}
Multiplying equation \eqref{dNLKG} with $m(u)$ and integrating it in the light cone, we get the Morawetz-type estimate with a space-time weight:
\EQ{ \label{Mor.est} 
 \int_S^T \int_{\substack{x\in\Om \\ |x|<t}} \frac{|x\dot u+t\na u|^2}{\la^3} + g(u)q + a\dot um(u) dxdt \lec E(u;0)}
for any $1<S<T$ and any solution $u$. See \cite{Noncoer,2Dsubcrit} for the details of the computation in the whole space $\R^N$. $  $
Without loss of generality, we may now choose 
\EQ{
 1\le R < S = 3R < T,}
where $R$ is the radius for effective damping given in \eqref{asm a prec}.\\
The damping term in \eqref{Mor.est} is dominated by the energy decrement 
\EQ{
 \int_S^T\int_{\Om} |a\dot um(u)| dxdt 
 \pt\le \|a\|_{L^\I_x}^{1/2}\sqrt{A(u;T)}\|m(u)\|_{L^2_{t,x}(S<t<T)}
 \pr\lec M^{1/2}\sqrt{A(u;T)}T^{1/2}\sqrt{E(u;0)}
 \pr\lec E(u;0)+MTA(u;T).}
The kinetic part in \eqref{Mor.est} is discarded by 
\EQ{
 \int_S^T\int_{\Om} \frac{|x\dot u|^2}{\la^3} dxdt 
 \pt\le \int_S^T \int_{\substack{x\in \Om \\ |x|<R}} \frac{R^2}{t^3}|\dot u|^2 dxdt
 + \int_S^T \int_{\substack{x\in\Om \\ |x|>R}} \frac{a(x)|\dot u|^2}{a_0(R+t)} dxdt
 \pr\lec E(u;0) + \frac{A(u;T)}{a_0R}.}
Using that $g\ge 0$ as well, we thus get 
\EQ{ \label{Mora bd}
 \int_S^T \int_{\substack{x\in\Om \\ |x|<t}} \frac{|\na u|^2}{t} + \frac{(t^2-|x|^2)g(u)}{t^3} dxdt 
 \lec E(u;0)[1+\mu(u;T)].}
where we put 
\EQ{ \label{def mu}
 \mu(u;T) := [MT+(a_0R)^{-1}]A(u;T)/E(u;0).}

Next, we fix a smooth cut-off function $\chi(x)\in C_0^\I(\R^N)$ such that $\chi(x)=1$ for $|x|<R$ and $\chi(x)=0$ for $|x|>2R$, and multiply \eqref{dNLKG} with $\chi u_n$. 
Integrating it in $\Om$, we get 
\EQ{ \label{L^2_tt}
 \p_t\LR{\dot u|\chi u}_{L^2_x}
 \pt= \int_{\Om} \chi\left[|\dot u|^2 - |\na u|^2 - |u|^2 - u f'(u)\right] dx
 \prq - \int_{\Om} [a\dot u\chi u + u\na u\cdot\na\chi] dx.}
Then by partial integration in $t$, and using the support property of $\chi$, 
\EQ{ \label{u^2_tt/t}
 \pt\int_S^T\int_{\Om} \chi\frac{|\dot u|^2}{t} dxdt
 \pr\le \int_S^T \frac{\p_t\LR{\dot u|\chi u}_{L^2_x}}{t} dt + \int_S^T\int_{\substack{x\in\Om \\ |x|<2R}} \frac{|\na u|^2+|u|^2+|uf'(u)|+|a\dot u|^2}{t} dx dt
 \pr\lec E(u;0) + MA(u;T) + \int_S^T \int_{\substack{x\in\Om \\ |x|<2R}}\frac{|\na u|^2+|u|^2+|uf'(u)|}{t}dxdt.}
For the last term we use the assumption on $f$ to replace it by 
\EQ{ \label{bd uf'}
 |uf'(u)| \le g(u) + 2f(u) \le (1+C_0)g(u)+C_0|u|^2.}
Then the $g$ term as well as the gradient term is treated by the Morawetz-type estimate. 
For the remaining $|u|^2$ part, we apply a weighted Sobolev inequality to the quadratic part of the Morawetz-type estimate, which implies
\EQ{\label{claim}
 \int_S^T \int_{\substack{x\in\Om \\ |x|<t}} \frac{|u|^p}{t} dxdt \lec E(u;0)^{p/2-1}\left[E(u;0)+\int_S^T\int_{\substack{x\in\Om \\ |x|<t}} \frac{|x\dot u+t\na u|^2}{\la^3}dxdt\right],}
for all $2+4/N\le p\le 2N/(N-2)$. 
The proof was given in \cite{Nak-rem} for the undamped NLKG, which works verbatim for the damped NLKG. 
Then applying H\"older and the support property of $\chi$, we get
\EQ{\label{u_^2/t bd}
 \int_{S}^T\int_{\Om} \chi\frac{|u|^2}{t} dxdt
 \pt\leq \left[\int_S^T\int_{\substack{x\in\Om \\ |x|<t}}\frac{|u|^p}{t}dxdt\right]^{2/p}\left[\int_{\substack{x\in\Om \\ |x|<2R}}\int_{S}^T \frac{dt}{t}dx\right]^{1-2/p}
 \pr\lec E(u;0)[1+\mu(u;T)]^{2/p} (R^N \log T)^{1-2/p}.}
On the other hand for $|x|>R$, $\dot u$ is estimated by
\EQ{ \label{u_t far}
 \int_S^T \int_{\substack{x\in\Om \\ |x|>R}} \frac{|\dot u|^2}{t} dxdt \le \int_S^T \int_{\Om} \frac{a(x)|\dot u|^2}{a_0 t} dxdt
 \le \frac{A(u;T)}{a_0 S}.} 
Putting together \eqref{u^2_tt/t}, \eqref{Mora bd}, \eqref{u_^2/t bd} and \eqref{u_t far}, we obtain 
\EQ{ \label{upper bd}
 \int_{S}^T \frac{\|\dot u\|_{L^2_x}^2}{t} dt 
 \pt\lec (1+C_0)E(u;0)[1+\mu(u;T)] 
 \pr+ E(u;0)[1+\mu(u;T)]^{2/p} (R^N \log T)^{1-2/p}.}
This is an upper bound. To get a lower bound, we go back to \eqref{L^2_tt} and replace $\chi$ by $1$. 
Then we obtain 
\EQ{\label{Id K}
 \p_t[\LR{\dot u|u}_{L^2_x}+\LR{au|u}_{L^2_x}/2] = \|\dot u\|_{L^2_x}^2 - K(u),}
where as before we denote
\EQ{ \label{low bd K defoc}
 K(u) \pn= \int_{\Om} |\na u|^2 + |u|^2 + uf'(u) dx
 \pt= E(u;t)-\|\dot u(t)\|_{L^2_x}^2 + \int_{\Om}  g(u)dx
 \pr\ge E(u;t)-\|\dot u(t)\|_{L^2_x}^2.}
Hence we get 
\EQ{
  2\|\dot u(t)\|_{L^2_x}^2 \ge E(u;t) + \p_t \LR{\dot u|u}_{L^2_x}+\LR{a\dot u|u}_{L^2_x},}
and integrating it by $dt/t$,  
\EQ{ \label{u^2 whole x}
 \int_{S}^T \frac{dt}{t}2\|\dot u(t)\|_{L^2_x}^2
 \pt\ge \int_{S}^T\frac{dt}{t}\left\{E(u;t)+ \p_t\LR{\dot u|u}_{L^2_x} + \LR{a\dot u|u}_{L^2_x}\right\}
 \pr\ge E(u;T)\log\frac{T}{S} + [\LR{\dot u|u/t}_{L^2_x}]_{S}^T + \int_{S}^T \LR{\dot u|u}_{L^2_x}\frac{dt}{t^2}
 \pr\pq\pq\pq -MA(u;T)-\int_S^T \|u\|_{L^2_x}^2\frac{dt}{t^2}
 \pr\ge E(u;T)\log\frac{T}{3R} - E(u;0)/R - MA(u;T).}
Combining it with \eqref{upper bd}, we obtain  
\EQ{ \label{combined est}
 E(u;T)\log\frac{T}{3R} \pt\lec (1+C_0)E(u;0)[1+\mu(u;T)] 
 \prq+ E(u;0)[1+\mu(u;T)]^{2/p} (R^N \log T)^{1-2/p}.}
Hence if  
\EQ{
 [1+MT+(a_0R)^{-1}]A(u;T) \le E(u;0)/2,}
then we have $\mu(u;T)\le 1/2$, $E(u;T)\ge E(u;0)/2$ and 
\EQ{
 \log T \lec 1 + C_0 + \log R  + (R^N\log T)^{1-2/p},}
which implies  
\EQ{
 \log T \lec 1 + C_0 + \log R + R^{N(p/2-1)} \lec 1 + C_0 + R^{N(p/2-1)}.}
In other words, there exists an absolute constant $C_*>0$ such that we have \eqref{energy decrease} and hence \eqref{exp dec} with 
\EQ{
 \log T = C_*(1 + C_0 + R^2), \pq \de=[1+MT+(a_0R)^{-1}]^{-1}/2,}
where we chose $p=2+4/N$. 

Finally we consider the weaker assumption \eqref{condition f2}. 
Then \eqref{bd uf'} is changed to 
\EQ{
 |uf'(u)| \le (1+C_0)g(u) + C_0(|u|^2 + |u|^q).}
where without loss of generality, we may assume that $q\ge 2+4/N$.  
Then to bound \eqref{u^2_tt/t} we need also \eqref{claim} with $p=q$, and consequently \eqref{upper bd} is modified to 
\EQ{ \label{upper bd2}
 \int_{S}^T \frac{\|\dot u\|_{L^2_x}^2}{t} dt 
 \pt\lec (1+C_0)E(u;0)[1+\mu(u;T)] + C_0E(u;0)^{q/2}[1+\mu(u;T)]
 \pr+ E(u;0)[1+\mu(u;T)]^{2/p} (R^N \log T)^{1-2/p},}
and so the decay rate in this case is given by 
\EQ{
 \log T = C_*(1 + C_0 + R^2 + C_0 E(u;0)^{q/2-1}), \pq \de=[1+MT+(a_0R)^{-1}]^{-1}/2.}
\qedsymbol

\section{Proof in the focusing case: Theorem \ref{Mainresult3}}
In this section, we prove the exponential decay in the focusing case, namely Theorem \ref{Mainresult3}. 
In order to proceed in the same way as in Section \ref{defoc sub}, we need the following three ingredients:
\begin{enumerate}
\item Upper bound on $E_F(u;t)$ in terms of $E(u;t)$. 
\item Lower bound on $K(u(t))$ in terms of $E(u;t)$. 
\item Upper bound on the space-time integral of $|g(u)|/\la$ in terms of the energy and Strichartz norms of the undamped equation. 
\end{enumerate}
\subsection{Variational properties}
The first two are well-known and coming from the variational characterization of $Q$. 
In fact, since $|uf'(u)|\ge(2+4/N)|f(u)|$, we have  
\EQ{ \label{free eng bd}
 (N+2)E(u;t) \ge NK(u(t)) + 2E_F(u;t) + \int N|\dot u|^2 dx > 2E_F(u;t),}
and for any $E_0\in(0,J(Q))$ there exists $\nu>0$ determined by $N$, $f$ and $E_0$ such that for any $\fy\in H^1(\R^N)$ satisfying $J(\fy)\le E_0$ and $K(\fy)>0$ we have 
\EQ{ \label{low bd K foc}
 K(\fy) \ge \nu \|\fy\|_{H^1}^2.}
For a proof of this lower bound, see \cite[Lemma 2.12]{IMN}. 
In the exponential case $f=f_{exp}$ ($N=2$), we have another important property \cite[Lemma 2.11]{IMN}
\EQ{ \label{exp sub}
 \|\na u(t)\|_{L^2_x}^2 + \|\dot u(t)\|_{L^2_x}^2 \le E(u;t) \le E_0 < m \le 1,}
which ensures that the solution stays in the subcritical regime for the Trudinger-Moser inequality 
\EQ{
 \|\na\fy\|_{L^2(\R^2)}<1 \implies \int_{\R^2}(e^{4\pi|\fy(x)|^2}-1) dx \lec \frac{\|\fy\|_{L^2(\R^2)}^2}{1-\|\na\fy\|_{L^2(\R^2)}^2}.}

\subsection{Global Strichartz bound from the scattering}
The integral bound on $|g(u)|/\la$ requires transfer of global Strichartz norms from the undamped equation. 
Specifically, we use the following Strichartz norms: 
\EQ{
 \|u\|_{St_T} := \|u\|_{L^{2_\star}_t(0,T;B^{1/2}_{2_\star,2}(\R^N))} + \|u\|_{L^{2+4/(N-1)}_t(0,T;B^{1/2}_{2+4/(N-1),2}(\R^N))},}
if $N\ge 3$, 
\EQ{
 \|u\|_{St_T} := \|u\|_{L^{2_\star}_t(0,T;B^{1/2}_{2_\star,2}(\R^2))} + \|u\|_{L^4_t(0,T;B^{1/4}_{\I,2}(\R^2))},}
if $N=2$, and 
\EQ{
 \|u\|_{St_T} := \|u\|_{L^6_t(0,T;B^{1/2}_{6,2}(\R))},}
if $N=1$, where $B^s_{p,q}$ denotes the inhomogeneous Besov space. 

The main result of \cite{IMN} asserts that under the assumptions of Theorem \ref{Mainresult3}, there exists a unique global solution $v(t,x):\R^{1+N}\to\R$ for the undamped NLKG \eqref{NLKG}, which satisfies 
\EQ{
 \sup_{t>0}E_F(v;t)\le A<\I, \pq  \|v\|_{St_\I} \le B<\I,} 
for some $A$ and $B$ determined by $N$, $f$ and $E_0$. 
Moreover, one has the following global perturbation result \cite[Lemma 4.5]{IMN}: {\it There exists $\e>0$, determined by $N$, $f$ and $E_0$, such that for any $T>1$ and any $u(t,x):[0,T]\times\R^N\to\R$ satisfying 
\EQ{
 E_F(u-v;0)^{1/2} + \|\ddot u-\De u + u + f'(u)\|_{L^1_t(0,T;L^2_x)} \le \e,}
we have 
\EQ{
 \|u-v\|_{St_T} \le C<\I,}
where $C$ is also determined by $N$, $f$ and $E_0$.}

We apply the above perturbation to the solution $u$ of the damped NLKG \eqref{dNLKG} with the same initial data $E_F(u-v;0)=0$. Since 
\EQ{
 \|a\dot u\|_{L^1_t(0,T;L^2_x)} \le M^{1/2}T^{1/2}A(u;T)^{1/2},}
we get 
\EQ{ \label{Stz bd}
 \|u\|_{St_T}\le C+B<\I,} 
provided that 
\EQ{
 A(u;T) \le (MT)^{-1}\e.}
Otherwise we get the exponential decay \eqref{exp rate} with $\de=(MT)^{-1}\e/E_0$. 

\subsection{Morawetz estimate from the Strichartz bound}
By using the Strichartz bound \eqref{Stz bd}, we can bound the nonlinear part in the Morawetz estimate 
\EQ{
 \int_1^T \int_{\R^N} \frac{|g(u)|}{\la} dx dt,}
as follows. Denote $\NN:=[(N+1)E_0]^{1/2}+C+B$ so that we have 
\EQ{
 E_F(u;t)^{1/2} + \|u\|_{St_T} \le \NN.} 
First we consider the case $N\ge 3$. For the $H^1$ critical power, we get by the H\"older, Sobolev and interpolation inequalities  
\EQ{
 \int_0^T\int_{\R^N} \frac{|u|^{2^\star}}{|x|} dxdt
  \pt\lec \||x|^{-1}\|_{L^{N,\I}_x} \|u\|^{2^\star}_{L^{2^\star}_t(0,T; L^{2^\star N/(N-1),2}_x)}
  \pr\lec \|u\|^{2^\star}_{L^{2^\star}_t(0,T;B^{1/2^\star}_{2^\star,2})}
  \pn\lec \|u\|_{L^\I_t(0,T;H^1_x)}^{\frac{4}{(N-1)(N-2)}} \|u\|_{St_T}^{\frac{2(N+1)}{N-1}} \lec \NN^{2^\star},}
where $L^{p,q}_x$ denotes the Lorentz space on $\R^N$. 
For the $L^2$ critical power, we have 
\EQ{
\int_0^T\int_{\R^N}\frac{|u|^{2_\star}}{|x|} dxdt
  \pt\lec \||x|^{-1}\|_{L^{N,\I}_x} \|u\|^{2_\star}_{L^{2_\star}_t(0,T; L^{(2N+4)/(N-1),2}_x)} 
  \pr\lec \|u\|^{2_\star}_{L^{2_\star}_t(0,T; B^{N/(2N+4)}_{2_\star,2})} \lec \|u\|_{St_T}^{2_\star}\lec \NN^{2_\star}.}
Next we consider the case $N=2$ with $f_p$, $p\ge 4=2_\star$. 
Then we have 
\EQ{
 \int_0^T \int_{\R^2} \frac{|u|^{p}}{|x|} dxdt
  \pt\lec \||x|^{-1}\|_{L^{2,\I}_x} \|u\|^{p}_{L^{p}_t(0,T; L^{2p,2}_x)}
  \pr\lec \|u\|^{p}_{L^{p}_t(0,T; B^{1/p}_{p,2})} 
  \lec \|u\|_{L^\I_t(0,T; B^0_{\I,2})}^{p-4} \|u\|_{L^4_t(0,T;B^{1/4}_{4,2})}^4 
  \pr\lec \|u\|_{L^\I_t(0,T;H^1)}^{p-4} \|u\|_{St_T}^{4} \lec \NN^p.}
For the case $N=1$ with $f_p$ and $p\ge 6=2_\star$, we have 
\EQ{
 \int_1^T \int_{\R} \frac{|u|^p}{\la} dxdt
  \le \|u\|_{L^\I(0,T;L^\I_x)}^{p-6} \|u\|_{L^6_t(0,T;L^6_x)}^6 \lec \NN^6.}
Finally we consider the exponential case $f=f_{exp}$ ($N=2$), which requires a bit more work. 
We start with H\"older as in the previous cases  
\EQ{ \label{est g exp}
 \int_0^T \int_{\R^2}  \frac{|g(u)|}{|x|}\;dxdt \lec \||x|^{-1}\|_{L^{2,\I}_x}\|g(u)\|_{L^1_t(0,T;L^{2,1}_x)},}
and to the $L^{2,1}_x$ norm we apply the following exponential estimate. 
Its $L^2_x$ version was proved in \cite{IMMN}. 
\begin{lem} \label{crit-nlst}
Let $g=uf'(u)-2f(u)$ and $f(u)=f_{exp}(u)$ as given in \eqref{f exp}. 
Then for any $A\in(0,1)$, there exist $\ka\in(0,1/8)$, and a
continuous increasing function $C_A:[0,\I)\to[0,\I)$, such that
for any $\fy \in H^1(\R^2)$ satisfying 
\EQ{ \label{H6}
  \|\nabla\fy\|_{L^2}\le A,}
we have,  
\EQ{ \label{NLE0}
 \|g(\fy)\|_{L^{2,1}} 
  \le C_A(\|\fy\|_{H^1})\|\fy\|_{C^{1/4-\ka}}^{4}\|\fy\|_{B^{\ka/2}_{2/\ka,2}}^4 +\|\fy\|_{L^{6,2}}^{6}.}
\end{lem}
Note that \eqref{H6} is satisfied by $u(t)$ thanks to the variational property \eqref{exp sub}, and the coefficient $C_A(\|u(t)\|_{H^1})$ is uniformly bounded. 
Thus we can continue the estimate \eqref{est g exp} using the above lemma, 
\EQ{
 \int_0^T \int \frac{|g(u)|}{|x|}dxdt
 \pt\lec \|u\|_{L^{2/\ka}_t(0,T;B^{\ka/2}_{2/\ka,2})}^4 \|u\|_{L^{4/(1-2\ka)}_t(0,T;C^{1/4-\ka})}^4
   +\|u\|_{L^{6}_t(0,T;B^{1/6}_{6,2})}^6
 \pr\lec \|u\|_{L^\I_t(0,T;H^1)}^{4}\|u\|_{L^4_t(0,T;B^{1/4}_{4,2})}^{8\ka}\|u\|_{L^4_t(0,T;B^{1/4}_{\I,2})}^{4-8\ka}
 \pr\pq+\|u\|_{L^\I_t(0,T;H^1)}^{2}\|u\|_{L^4_t(0,T;B^{1/4}_{4,2})}^{4}
 \pr\lec \NN^8+\NN^6.}

It remains to prove the lemma. It is based on a sharp logarithmic inequality proved in \cite[Theorem 1.3]{DlogSob}.
\begin{lem} \label{Hmu}
Let $0<\alpha<1$. For any $\lambda>\frac{1}{2\pi\alpha}$ and any $0<\mu\leq1$, 
there exists a constant $C_{\lambda}>0$ such that, for any function $u\in H^1(\R^2)\cap{\mathcal C}^\alpha(\R^2)$
\begin{equation}
\label{H-mu} \|u\|^2_{L^\infty}\leq
\lambda\|u\|_{H_\mu}^2\log\left(C_{\lambda} +
\frac{8^\alpha\mu^{-\alpha}\|u\|_{{\mathcal
C}^{\alpha}}}{\|u\|_{H_\mu}}\,\,\,\right),
\end{equation}
where $\|u\|_{H_\mu}^2:=\|\nabla u\|_{L^2}^2+\mu^2\|u\|_{L^2}^2$.
\end{lem}
The above logarithmic inequalities constitute a refinement (with respect to the best constant) of that appeared in Br\'ezis-Gallouet \cite{BrGa}. 
\begin{proof}[Proof of Lemma \ref{crit-nlst}]
Without loss of generality, we may assume that $A\in(1/2,1)$. Let
$\fy\in H^1$ satisfying \eqref{H6}. We define $\mu>0$ by 
\EQ{
 2\mu^2(1+\|u\|_{H^1_x}^2) = 1-A^2,}
so that we have 
\EQ{\label{A1} 
 \|\fy\|_{H_\mu}^2\leq A_1:=(1+A^2)/2 < 1. } 
Now we can choose $\ka\in(0,1/8)$ and $\la>2/\pi$, depending only on $A_1$,
such that 
\EQ{
 2\pi\la(1/4-\ka) > 1 =  2\pi\la (1/4+\ka/2) A_1^2. }
For instance, we can choose $\ka=(1-A_1^2)/8$. 
Note that 
\EQ{
 |g(u)| \lec |u|^4(e^{4\pi|u|^2}-1).}
In particular, if $\|\fy\|_{L^\infty}\leq A_1$, then $|g(\fy)| \lec |\fy|^6$ and thanks to Sobolev embedding
\EQ{ \label{small Linfty} 
 \|g(\fy)\|_{L^{2,1}} \lec \|\fy\|_{L^{12,6}}^6
 \lec \|\fy\|_{L^{12,2}}^6 \lec \|\fy\|_{B^{1/6}_{6,2}}^6.} 
If $\|\fy\|_{L^\I}>A_1$, then we have 
\EQ{ 
 \|g(\fy)\|_{L^{2,1}}
 & \le \|e^{4\pi|\fy|^2}-1\|_{L^1}^{1/2-\ka}
 \|e^{4\pi|\fy|^2}-1\|_{L^\I}^{1/2+\ka}
  \|\fy \|_{L^{4/\ka,4/(1/2+\ka)}}^4\\
 & \le \|e^{4\pi|\fy|^2}-1\|_{L^1}^{1/2-\ka}
  \|e^{4\pi|\fy|^2}-1\|_{L^\I}^{1/2+\ka}
  \|\fy \|_{L^{4/\ka,2}}^4.}
The first factor on the right hand side in the above inequality is
bounded by the sharp Trudinger-Moser inequality. Since $\|\fy\|_{C^{1/4-\ka}}>A>1/2$ and $C_\la\ge 1$, then the  second factor is bounded by using the logarithmic inequality \eqref{H-mu}
\EQ{ \label{exponential estimate}
\|e^{4\pi|\fy|^2}-1\|_{L^\I}^{1/2+\ka} 
 \pt\le e^{2\pi(1+2\ka)\|\fy\|_{L^\I}^2}
  \le \left[C_\la+\frac{(8/\mu)^{1/4}\|\fy\|_{C^{1/4-\ka}}}{\|\fy\|_{H_\mu}}\right]^{2\pi(1+2\ka)\la\|\fy\|_{H_\mu}^2} 
 \pr\le [C_\la+(8/\mu)^{1/4}\|\fy\|_{C^{1/4-\ka}}/A_1]^{2\pi(1+2\ka)\la A_1^2}
 \pr\le (8/\mu)[4C_\la \|\fy\|_{C^{1/4-\ka}}]^{4}.}
Thus, using Sobolev embedding we obtain 
\EQ{
 \|g(\fy)\|_{L^{2,1}}
  \lec \|\fy\|_{B^{1/6}_{6,2}}^6
   + C_A(\|\fy\|_{H^1})\|\fy\|_{C^{1/4-\ka}}^{4} \|\fy\|_{B^{\ka/2}_{2/\ka,2}}^4. }
\end{proof}


\subsection{Concluding the proof of Theorem \ref{Mainresult3}.}
Thus we have so far either 
\EQ{ \label{large A foc}
 A(u;T) \ge (MT)^{-1}\e,}
or 
\EQ{
 \int_1^T \int_{\R^N} \frac{|g(u)|}{\la} dx dt \le C(\NN)E(u;0)<\I,}
where $C(\NN)>0$ is a certain uniform bound in terms of $E_0$, $f$ and $N$. 
In the latter case, we obtain from \eqref{Mor.est} in the same way as in Section \ref{defoc sub}, using the free energy bound \eqref{free eng bd} as well, 
\EQ{
 \pt\int_S^T \int_{|x|<t} \frac{|\na u|^2 + |x\dot u+t\na u|^2}{\la} + \frac{|g(u)|}{\la} dxdt 
  \prq\lec E(u;0)[1+\mu(u;T)+C(\NN)],}
for $1\le R<S=3R<T$, where $\mu$ is as defined in \eqref{def mu}. 
Then proceeding in the same way as before, we obtain
\EQ{ \label{upper bd foc}
 \int_{S}^T \frac{\|\dot u\|_{L^2_x}^2}{t} dt 
 \pt\lec (1+C_0)E(u;0)[1+\mu(u;T)+C(\NN)]  
 \pr+ E(u;0)[1+\mu(u;T)+C(\NN)]^{2/p} (R^N \log T)^{1-2/p},}
instead of \eqref{upper bd}. 
Further tracking the same way, we replace \eqref{low bd K defoc} by the following bound due to \eqref{low bd K foc}
\EQ{
 K(u) \ge \nu\|u(t)\|_{H^1}^2 \ge \nu[E(u;t)-\|u(t)\|_{L^2_x}^2],}
hence we have 
\EQ{
 (1+\nu)\|\dot u(t)\|_{L^2_x}^2 \ge \nu E(u;t) + \p_t\LR{\dot u|u}_{L^2_x} + \LR{a\dot u|u}_{L^2_x}.}
Integrate it by $dt/t$, we obtain the following instead of \eqref{u^2 whole x}:
\EQ{
 \int_{S}^T \frac{dt}{t}(1+\nu)\|\dot u(t)\|_{L^2_x}^2
  \ge \nu E(u;T)\log\frac{T}{3R} - E(u;0)/R - MA(u;T).}
Combining it with \eqref{upper bd foc}, we obtain 
\EQ{
 \nu E(u;T)\log\frac{T}{3R} \pt\lec (1+C_0)E(u;0)[1+\mu(u;T)+C(\NN)]  
 \pr+ E(u;0)[1+\mu(u;T)+C(\NN)]^{2/p} (R^N \log T)^{1-2/p}.}
Hence if 
\EQ{
 [1+MT+(a_0R)^{-1}]A(u;T) \le E(u;0)/2,}
then we have $\mu\le 1/2$, $E(u;T)\ge E(u;0)/2$ and 
\EQ{
 \nu \log T \lec \log R + (1+C_0)[1+C(\NN)] + [1+C(\NN)]^{2/p}(R^N \log T)^{1-2/p},}
which implies that 
\EQ{
 \log T \lec \nu^{-1}(1+C(\NN))(1+C_0+R^2),}
where we have chosen $p=2_\star=2+4/N$. 
Thus in conclusion there exists an absolute constant $C_*>0$ such that we have \eqref{exp dec} with 
\EQ{
 \pt \log T = C_*\nu^{-1}(1+C(\NN))(1+C_0+R^2), 
 \pr \de=\min([1+MT+(a_0R)^{-1}]^{-1}/2,(MT)^{-1}\e/E_0),}
where the last factor takes care of the case \eqref{large A foc}.  
\qedsymbol 



\end{document}